\documentclass[11pt]{amsart}
\usepackage{amscd,amsxtra,amssymb, bbm}

\usepackage[dvips]{graphics}
\usepackage{epsfig,graphicx}

\input xy
\xyoption{all}
\xyoption{arc}
\SelectTips{cm}{10}

\newtheorem{theorem}{Theorem}
\newtheorem{corollary}[theorem]{Corollary}
\newtheorem{lemma}[theorem]{Lemma}

\theoremstyle{definition}
\newtheorem{definition}[theorem]{Definition}
\newtheorem{remark}[theorem]{Remark}

\newtheorem{example}[theorem]{Example}

\DeclareMathOperator{\rk}{\mathrm{rk}}

\DeclareMathOperator{\VB}{\mathsf{VB}}

\DeclareMathOperator{\Tri}{\mathsf{Tri}}

\DeclareMathOperator{\Hom}{\mathsf{Hom}}

\DeclareMathOperator{\Spec}{\mathsf{Spec}}

\newcommand{\kk}{\boldsymbol{k}}
\newcommand{\CC}{\mathbb{C}}
\newcommand{\ZZ}{\mathbb{Z}}

\newcommand{\FF}{\mathbb{F}}

\newcommand{\PP}{\mathbb{P}}

\newcommand{\kB}{\mathcal{B}}

\newcommand{\kF}{\mathcal{F}}

\newcommand{\kI}{\mathcal{I}}
\newcommand{\kO}{\mathcal{O}}

\newcommand{\kV}{\mathcal{V}}

\newcommand{\kT}{\mathcal{T}}

\newcommand{\mm}{\mathsf{m}}

\newcommand{\lar}{\longrightarrow}

\begin{document}

\title[Frobenius morphism and vector bundles]{Frobenius morphism and vector bundles on cycles of projective lines}
\author{Igor Burban}
\address{
Mathematical  Institute,
Weyertal 86--90, 
D--50931 Cologne,
Germany
}
\email{burban@math.uni-koeln.de}


\begin{abstract}
In this paper we describe the action of the Frobenius morphism on the indecomposable
vector bundles on cycles of projective lines. This gives an answer to a  question of Paul Monsky, which
appeared in his study of the Hilbert--Kunz theory for plane cubic curves.
\end{abstract}

\maketitle

This article  arose  as an  answer to  a question posed by Paul Monsky in his study of the Hilbert--Kunz theory
for   plane cubic curves \cite{Monsky}.
Let $\kk$ be an algebraically closed  field of characteristic $p >0$   and $E$ be
a projective curve of arithmetic genus one over $\kk$. We are interested in an explicit description of the  action
of the Frobenius morphism on the  indecomposable vector bundles on $E$. In
the case of elliptic curves, this problem has been  solved  by Oda \cite[Theorem 2.16]{Oda}.
In this article we deal with  the case when $E$ is an irreducible rational
nodal curve or a cycle of projective lines.

\medskip

We start by  recalling the general technique used in a  study of
vector bundles  on singular projective curves, see \cite{DrozdGreuel, BBDG, Thesis}.
Let $X$ be a reduced singular (projective)  curve over $\kk$, $\pi:  \widetilde{X} \to  X$
its normalization and  $\kI :=
{\mathcal Hom}_\kO\bigl(\pi_*(\kO_{\widetilde{X}}), \kO\bigr) =
{\mathcal A}nn_\kO\bigl(\pi_*(\kO_{\widetilde{X}})/\kO\bigr)$
the conductor ideal sheaf.
Denote  by $\eta: Z = V(\kI) \rightarrow X$ the
closed artinian subscheme   defined by $\kI$
(its topological support is precisely the singular locus of $X$)  and by
$\tilde\eta: \widetilde{Z} \rightarrow  \widetilde{X}$ its preimage in
$\widetilde{X}$, defined by the Cartesian  diagram
\begin{equation}\label{E:keydiag}
\begin{array}{c}
\xymatrix
{\widetilde{Z} \ar[r]^{\tilde{\eta}} \ar[d]_{\tilde{\pi}}
& \widetilde{X} \ar[d]^\pi \\
Z \ar[r]^\eta & X.
}
\end{array}
\end{equation}

\begin{definition}
  The category of triples $\Tri(X)$ is defined as follows.
  \begin{itemize}
  \item Its objects are triples
    $\bigl(\widetilde\kF, \kV, \widetilde \mm\bigr)$, where
    $\widetilde\kF \in \VB(\widetilde{X})$, $\kV \in \VB(Z)$ and
    $$ \mm: \tilde{\pi}^*\kV \lar \tilde{\eta}^*\widetilde\kF$$ is an
    isomorphism of $\kO_{\widetilde{Z}}$--modules, called the \emph{gluing map}.
  \item The set of morphisms
    $\Hom_{\Tri(X)}\bigl((\widetilde\kF_1, \kV_1, {\mm}_1),
    (\widetilde\kF_2, \kV_2, {\mm}_2)\bigr)$ consists of all pairs
    $(F,f)$, where $F: \widetilde\kF_1 \to  \widetilde\kF_2$ and
    $f: \kV_1 \to \kV_2$
    are morphisms of vector bundles such that the following diagram is
    commutative
    $$
    \xymatrix
    {\tilde{\pi}^*\kV_1 \ar[r]^{\mm_1}\ar[d]_{\tilde{\pi}^*(f)} &
      \tilde{\eta}^*\widetilde\kF_1
      \ar[d]^{\tilde{\eta}^*(F)} \\
      \tilde{\pi}^*\kV_2 \ar[r]^{\mm_2} & \tilde{\eta}^*\widetilde\kF_2.
    }
    $$
  \end{itemize}
\end{definition}

\begin{theorem}[Lemma 2.4 in \cite{DrozdGreuel}, see also
  Theorem 1.3 in \cite{Thesis}]\label{thm:Drozd-Greuel}
  Let $X$ be a reduced curve over $\kk$. Then the functor
     $\mathbb{F}: \VB(X) \rightarrow  \Tri(X)$ assigning  to a  vector
bundle $\kF$ the triple  \linebreak
$(\pi^*\kF, \eta^*\kF, {\mm}_{\kF})$, where
${\mm}_{\kF}: \tilde{\pi}^*(\eta^*\kF) \rightarrow \tilde\eta^*(\pi^*\kF)$ is
the canonical isomorphism,  is  an equivalence of categories.
\end{theorem}

\begin{remark}
In the partial case when $X$ is a configuration of projective lines intersecting transversally,
the above theorem also follows  from a more general result of Lunts \cite{Lunts}.
\end{remark}

For a ringed space $(Y, \kO_Y)$ over  $\kk$ we denote by
$\varphi_Y$ the Frobenius morphism $(Y, \kO_Y) \rightarrow  (Y, \kO_Y)$. Then for
an open set $U \subset Y$ the ring  homomorphism $\varphi_Y(U): \kO_Y(U) \to \kO_Y(U)$ is given by the formula
 $\varphi_Y(f) = f^p$,  $f \in \kO_Y(U)$.
For simplicity, we shall frequently omit the subscript in the notation
of the Frobenius map.

\begin{definition}
  Let endofunctor $\PP: \Tri(X) \lar \Tri(X)$ be  defined as follows. For an
   object $\kT = (\widetilde\kF, \kV, \mm)$  of the category $\Tri(X)$ we set
 $\PP(\kT) := \bigl(\varphi^*\widetilde\kF,  \varphi^*\kV, \mm^\varphi\bigr)$, where
 the gluing map $\mm^\varphi$ is determined via the commutative diagram
 $$
 \xymatrix{
 \varphi^* \tilde\pi^* \kV \ar[rr]^{\varphi^*(\mm)} \ar[d]_{\mathsf{can}} & & \varphi^* \tilde\eta^* \widetilde\kF \ar[d]^{\mathsf{can}}\\
 \tilde\pi^* \varphi^* \kV \ar[rr]^{\mm^\varphi} & & \tilde\eta^* \varphi^* \widetilde\kF,
 }
$$
and both vertical maps are canonical isomorphisms.
\end{definition}

\begin{lemma}\label{L:Frob-and-Triples}
Consider the following diagram of categories
and functors:
$$
\xymatrix{
\VB(X) \ar[rr]^{\FF} \ar[d]_{\varphi_X^*} & & \Tri(X) \ar[d]^{\PP}\\
\VB(X) \ar[rr]^{\FF} & & \Tri(X),
}
$$
Then there exists an isomorphism  $\PP \circ \FF \rightarrow \FF \circ \varphi_X^*$.
\end{lemma}

\begin{proof} Let $\kF$ be a vector bundle on $X$. Then the  canonical isomorphisms
$\varphi^* \tilde\eta^* \kF \rightarrow \tilde\eta^* \varphi^*\kF$ and
$\varphi^* \pi^* \kF \rightarrow \pi^* \varphi^*\kF$ induce the commutative diagram
$$
 \xymatrix{
 \tilde\pi^* \varphi^* \tilde\eta^* \kF \ar[rr]^{\mm_\kF^\varphi} \ar[d]_{\mathsf{can}} & & \tilde\eta^* \varphi^* \pi^* \kF \ar[d]^{\mathsf{can}} \\
\tilde\pi^*  \tilde\eta^* \varphi^* \kF \ar[rr]^{\mm_{\varphi^*\kF}} & & \tilde\eta^*  \pi^* \varphi^* \kF,
 }
$$
which yields  the desired isomorphism of functors.
\end{proof}

Next, we need a description of the action of the Frobenius map on the vector
bundles on a projective line.
Let $(z_0, z_1)$ be coordinates on $V = \CC^2$. They induce homogeneous
coordinates $(z_0:z_1)$ on
$\PP^1 = \PP^1(V) = \bigl(V \setminus\{0\}\bigr)/\sim$, where $v \sim \lambda v$ for all $v \in V$
and
$\lambda \in \CC^*$.
We set
$U_0 = \bigl\{(z_0: z_1)| z_0 \ne 0\bigr\}$ and $U_\infty  = \bigl\{(z_0: z_1)| z_1 \ne 0\bigr\}$
and put $0 := (1: 0)$, $\infty := (0: 1)$, $z = z_1/z_0$ and $w = z_0/z_1$.
So, $z$ is a   coordinate in a  neighbourhood of $0$.
If $U = U_0 \cap U_\infty$ and $w=1/z$ is used as a coordinate on
$U_{\infty}$, then the transition function of the line bundle $\kO_{\PP^1}(n)$ is
\begin{equation}\label{E:actionofFrob}
 U_0 \times \mathbb{C} \supset  U \times \mathbb{C}
\xrightarrow{(z,v) \mapsto \left(\frac{1}{z}, \frac{v}{z^n}\right)}
U \times \mathbb{C} \subset  U_\infty \times \mathbb{C}.
\end{equation}

\medskip
\noindent
The proof of the following lemma is straightforward.

\begin{lemma}
For any $n \in \ZZ$ we have:
$\varphi^*\bigl(\kO_{\PP^1}(n)\bigr) \cong  \kO_{\PP^1}(np)$.
\end{lemma}

\noindent
Next, recall the following classical result on vector bundles on a projective line.

\begin{theorem}[Birkhoff--Grothendieck]\label{T:BirkhoffGrothendieck}
Any vector bundle
$\widetilde\kF$ on $\PP^1$   splits into a direct sum of line  bundles:
\begin{equation}\label{E:BirkGroth}
\widetilde\kF \cong \bigoplus_{l \in \mathbb{Z}} \kO_{\PP^1}(l)^{m_l}.
\end{equation}
\end{theorem}

\noindent
Now  assume that $E$ is an \emph{irreducible}  rational  \emph{nodal}  curve.
Note that by  the definition of being nodal we have:  $Z =  \Spec(\kk)$.

\begin{example}
The plane cubic curve $E \subset \mathbb{P}^2$,  given by the homogeneous equation $x^3+y^3 -xyz = 0$,
is an irreducible  rational curve with a nodal singularity  at $(0:0:1)$.
\end{example}

\noindent
Theorem  \ref{T:BirkhoffGrothendieck}
implies that for an object $(\widetilde{\kF}, \kV, \widetilde{\mm})$ of  $\Tri(E)$ with $\rk(\widetilde{\kF}) = n$, we have
\[\widetilde{\kF} = \bigoplus\limits_{l \in \mathbb{Z}} \kO_{\PP^1}(l)^{m_l}
\quad \text{ and } \quad \quad \kV \cong \kO_{Z}^n, \quad
\text { where } \sum\limits_{l \in \mathbb{Z}} m_l = n.\]
The vector bundle  $\kV$ is free, because $Z$ is artinian.  In order to describe
the morphism $\widetilde{\mm}$ in the terms of matrices, some additional choices have to be made.

Recall that  the vector bundle $\kO_{\PP^1}(-1)$  is isomorphic to the sheaf of sections
of  the so-called tautological line bundle
$$\bigl\{(l, v) | \, v \in l \bigr\} \subset \PP^1(V) \times V = \kO_{\PP^1}^2.$$
The choice of coordinates on $\PP^1$ fixes two distinguished elements, $z_0$
and $z_1$, in the vector space $\Hom_{\PP^1}\bigl(\kO_{\PP^1}(-1), \kO_{\PP^1}\bigr)$:
$$
\xymatrix
{ \PP^1 \times \mathbb{C}^2 \ar[rd] \ar@{<-^{)}}[r] &
                   \kO_{\PP^1}(-1) \ar[d] \ar[r]^{z_i} &
                   \PP^1 \times \mathbb{C} \ar[ld] \\
& \PP^1 &
}
$$
where $z_i$ maps $\bigl(l, (v_0, v_1)\bigr)$ to $(l, v_i)$ for  $i = 0,1$.
It is clear that the section $z_0$ vanishes at $\infty$ and $z_1$
vanishes at $0$.
In what follows,  we shall assume that the coordinates on the normalization $\widetilde{E}$ are chosen
in such a way that $\Spec(\kk \times \kk) \cong \widetilde{Z} = \pi^{-1}(Z) =
\bigl\{0, \infty\bigr\}$.

\begin{definition}\label{D:Trivialization}
For any $l \in \ZZ$ we define the isomorphism $\xi_l: \tilde\eta^*\bigl(\kO_{\PP^1}(l)\bigr)
\rightarrow \kO_{\widetilde{Z}}$ by the formula $\xi_l(s) = \bigl(\frac{s}{z_0^l}(0), \frac{s}{z_1^l}(\infty)\bigr)$, where $s$ is  an arbitrary  local section of the line bundle $\kO_{\PP^1}(l)$.
Hence, for any vector bundle $\widetilde\kF$ of rank $n$ on $\PP^1$ given by the formula
(\ref{E:BirkGroth}), we have the induced isomorphism $\xi_{\widetilde\kF}: \tilde\eta^*\widetilde\kF \rightarrow
\kO_{\widetilde{Z}}^n$.
\end{definition}

Let $(\widetilde\kF, \kO_Z^n, \mm)$ be an  object in  the category of triples $\Tri(E)$. Note
that we have a unique morphism  $M({\mm})$ making the following  diagram
commutative:
\begin{equation}\label{E:defgluematr}
\begin{array}{c}
\xymatrix{
\tilde\pi^* \kO_{Z}^n \ar[rr]^{\mm} \ar[d]_{\mathsf{can}} & & \tilde\eta^* \widetilde\kF \ar[d]^{\xi_{\widetilde\kF}} \\
\kO_{\widetilde{Z}}^n \ar[rr]^{M({\mm})} & & \kO_{\widetilde{Z}}^n,
}
\end{array}
\end{equation}
where the first vertical map is the canonical isomorphism. Moreover,
$M({\mm})$ is given by a pair of invertible $(n \times n)$ matrices
$M(0)$ and $M(\infty)$ over the field $\kk$.
Applying to
 (\ref{E:defgluematr}) the functor $\varphi^*$, we get the following
commutative diagram:
\begin{equation}
\begin{array}{c}
\xymatrix{
\tilde\pi^* \varphi^* \kO_{Z}^n \ar[rr]^{\mm^\varphi} \ar@/_40pt/[ddd]_{\mathsf{can}} & & \tilde\eta^* \varphi^* \widetilde\kF \ar@/^40pt/[ddd]^{\xi_{\varphi^*\widetilde\kF}} \\
\varphi^* \tilde\pi^*  \kO_{Z}^n \ar[rr]^{\varphi^*(\mm)} \ar[u]^{\mathsf{can}} \ar[d]_{\mathsf{can}} & & \varphi^* \tilde\eta^*  \widetilde\kF \ar[u]_{\mathsf{can}} \ar[d]^{\varphi^*(\xi_{\widetilde\kF})} \\
\varphi^*\bigl(\kO_{\widetilde{Z}}^n\bigr)  \ar[rr]^{\varphi^*(M({\mm}))} \ar[d]_{\mathsf{can}} & &
\varphi^*\bigl(\kO_{\widetilde{Z}}^n\bigr) \ar[d]^{\mathsf{can}} \\
\kO_{\widetilde{Z}}^n \ar[rr]^{M({\mm}^\varphi)} & & \kO_{\widetilde{Z}}^n.
}
\end{array}
\end{equation}

\begin{corollary}\label{C:actionFrob}
{\rm
Let $E$ be an irreducible rational nodal  curve over a field $\kk$ of characteristic $p >0 $ and $\kF$ be a vector bundle corresponding to the triple
$(\widetilde\kF, \kO_Z^n, \mm)$, where $\widetilde\kF \cong \oplus_{l \in \mathbb{Z}} \kO_{\PP^1}(l)^{m_l}$
and $\mm$ is given by a pair of matrices
$$
M(0) =
\left(
\begin{array}{cccc}
a_{11} & a_{12} & \dots & a_{1n} \\
a_{21} & a_{22} & \dots & a_{2n} \\
\vdots & \vdots & \ddots & \vdots \\
a_{n1} & a_{n2} & \dots & a_{nn}
\end{array}
\right)
\quad
\mbox{\rm and}
\quad
M(\infty) =
\left(
\begin{array}{cccc}
b_{11} & b_{12} & \dots & b_{1n} \\
b_{21} & b_{22} & \dots & b_{2n} \\
\vdots & \vdots & \ddots & \vdots \\
b_{n1} & b_{n2} & \dots & b_{nn}
\end{array}
\right).
$$
Then the vector bundle $\varphi^*\kF$ is given by the triple
$(\varphi^* \widetilde\kF, \kO_Z^n, \mm^\varphi)$, where $\varphi^*\widetilde\kF \cong \oplus_{l \in \mathbb{Z}} \kO_{\PP^1}(lp)^{m_l}$
and $\mm^\varphi$ corresponds to the pair of matrices
$$
\left(
\begin{array}{cccc}
a_{11}^p & a_{12}^p & \dots & a_{1n}^p \\
a_{21}^p & a_{22}^p & \dots & a_{2n}^p \\
\vdots & \vdots & \ddots & \vdots \\
a_{n1}^p & a_{n2}^p & \dots & a_{nn}^p
\end{array}
\right)
\quad
\mbox{\rm and}
\quad
\left(
\begin{array}{cccc}
b_{11}^p & b_{12}^p & \dots & b_{1n}^p \\
b_{21}^p & b_{22}^p & \dots & b_{2n}^p \\
\vdots & \vdots & \ddots & \vdots \\
b_{n1}^p & b_{n2}^p & \dots & b_{nn}^p
\end{array}
\right).
$$
}
\end{corollary}

\begin{remark}\label{R:classificDG}
Recall, that the  indecomposable vector bundles on an irreducible  nodal rational
 curve $E$ over an algebraically closed   field $\kk$ are described
by  the following data:
\begin{itemize}
\item a non-periodic sequence of integers $\mathbbm{d} = (d_1, \dots, d_l)$,
\item  a  positive integer $m$,
\item a continuous parameter $\lambda \in \kk^*$,
\end{itemize}
 see  \cite[Theorem 2.12]{DrozdGreuel} or  \cite[Section 3]{BBDG}.
 The corresponding indecomposable vector bundle
  $\kF = \kB\bigl(\mathbbm{d}, m,  \lambda\bigr)$ has rank
  $lm$. By the definition (see e.g. \cite[Algorithm 1]{BBDG}) the corresponding
  triple  $\PP(\kF) \cong
 \Bigl(\widetilde\kF, \kV, \bigl(M(0), M(\infty)\bigr)\Bigr)$ is the following:
 $\widetilde\kF = \kO_{\PP^1}(d_1)^{m} \oplus \dots \oplus \kO_{\PP^1}(d_l)^{m}$, $\kV = \kO_Z^{lm}$
 and the gluing matrices are
$$
M(0) = \mathbbm{1}_{ml \times ml} \quad \mbox{and} \quad
 M(\infty) =
 \left(
 \begin{array}{ccccc}
 0 & \mathbbm{I} & 0 & \dots & 0 \\
 0 & 0 & \mathbbm{I} & \dots & 0 \\
 \vdots & \vdots & \ddots & \ddots & \vdots \\
 0 & 0 & 0  & \ddots  & \mathbbm{I}  \\
 \mathbbm{J} & 0 & 0 & \dots & 0
 \end{array}
 \right),
 $$
 where $\mathbbm{I} = \mathbbm{1}_{m \times m}$ is the identity matrix of size $m$  and $\mathbbm{J}$
 is   the Jordan block $J_m(\lambda)$.
 \qed
 \end{remark}

\begin{theorem} Let $E$ be an irreducible nodal rational curve over an algebraically closed
 field
$\kk$ of characteristic $p >0$,
$\mathbbm{d} = (d_1, d_2, \dots, d_l)$ be a non-periodic sequence
of integers, $m \in \mathbb{N}$ and $\lambda \in \kk^*$.
Let $\kF = \kB\bigl(\mathbbm{d}, m, \lambda \bigr)$ be the corresponding indecomposable
vector bundle on $E$. Then we have:
\begin{equation}\label{E:VBactionFrob}
\varphi^* \kF  \cong
\kB\bigl((p d_1, p d_2, \dots, p d_l), m, \lambda^p\bigr).
\end{equation}
In particular,   the vector bundle $\varphi^*\kF$ is  indecomposable.
\end{theorem}

\begin{proof}
It is a direct consequence of Theorem \ref{thm:Drozd-Greuel}, Corollary \ref{C:actionFrob} and
Remark \ref{R:classificDG}.
\end{proof}

\begin{remark}
The same argument literally applies to the case,  when $E$ is a cycle of projective lines.
In particular, an analogous  formula (\ref{E:VBactionFrob}) holds in that case, too. Note that
in the case of elliptic curves it is in general \emph{not true} that the pull-back
of an indecomposable vector bundle under the Frobenius morphism is again indecomposable \cite[Theorem 2.16]{Oda}.
\end{remark}

\medskip
\noindent
\emph{Acknowledgement}. This work  was  supported by the DFG
project Bu-1866/2-1.


\begin{thebibliography}{99}
\bibitem{BBDG}
L.~Bodnarchuk, I.~Burban, Yu.~Drozd and  G.-M.~Greuel,
\emph{Vector bundles and torsion free sheaves on degenerations of elliptic curves},
Global Aspects of Complex Geometry, 83--129, Springer  (2006).


\bibitem{Thesis} I.~Burban, \textit{Abgeleitete Kategorien und Matrixprobleme},
PhD Thesis, Kaiserslautern 2003, available at
\textsf{https://kluedo.ub.uni-kl.de/files/1434/phd.pdf}.




  \bibitem{DrozdGreuel}
Yu.~Drozd, G.-M.~Greuel, \textit{Tame and wild projective curves and
classification of vector bundles},   J. Algebra  \textbf{246}  (2001),  no.
\textbf{1}, 1--54.

\bibitem{Lunts}
V.~Lunts, \emph{Coherent sheaves on configuration schemes},
J. Algebra  \textbf{244}    no. \textbf{2} (2001), 379--406.

\bibitem{Monsky} P.~Monsky, \emph{Hilbert--Kunz theory for nodal cubics, via sheaves},
J.~Algebra \textbf{346}, no.~\textbf{1} (2011), 180--188.


\bibitem{Oda} T.~Oda,
\emph{Vector bundles on an elliptic curve},
Nagoya Math. J. \textbf{43} (1971), 41--72.



\end{thebibliography}
\end{document}